\documentclass{article}
\usepackage[margin=1in]{geometry}

\AtEndDocument{\bigskip{\footnotesize%
  \textsc{Department of Applied Mathematics, Delhi Technological University, Delhi–110042, India} \par
  \textit{E-mail address:} \texttt{spkumar@dtu.ac.in} \par
  \addvspace{\medskipamount}
  \textsc{Department of Applied Mathematics, Delhi Technological University, Delhi–110042, India} \par  
  \textit{E-mail address:} \texttt{suryagiri456@gmail.com} \par
}}
\usepackage{graphicx}
\usepackage{hyperref}
\usepackage[caption = false]{subfig}
\usepackage{amsthm}
\usepackage{amssymb}
\usepackage{mathrsfs}
\usepackage{mathtools}
\usepackage{enumerate}
\usepackage{amssymb}
\usepackage{wrapfig}
\usepackage{float}
\allowdisplaybreaks

\usepackage{thmtools}
\declaretheorem[numbered=no,
name=Theorem A]{theorem A}

\declaretheorem[numbered=no,
name=Theorem B]{theorem B}

\declaretheorem[numbered=no,
name=Theorem C]{theorem C}

\declaretheorem[numbered=no,
name=Theorem D]{theorem D}

\declaretheorem[numbered=no,
name=Theorem E]{theorem E}

\usepackage{geometry}
\numberwithin{equation}{section}
\DeclareMathOperator{\RE}{Re}

\theoremstyle{plain}
\usepackage{lipsum}

\newtheorem{theorem}{Theorem}[section]
\newtheorem{corollary}[theorem]{Corollary}

\theoremstyle{definition}
\newtheorem{definition}[theorem]{Definition}
\theoremstyle{remark}
\newtheorem{remark}{Remark}[section]

\makeatother

\usepackage{authblk}
\begin{document}
\title{Toeplitz Determinants in One and Higher Dimensions}
\author{Surya Giri and S. Sivaprasad Kumar}


\date{}


	

\maketitle	
	
\begin{abstract}
    In this study, we derive the sharp bounds of certain Toeplitz determinants whose entries are the coefficients of holomorphic functions belonging to  a class defined on the unit disk $\mathbb{U}$. Further, these results are extended to a class of holomorphic functions on the unit ball in a complex Banach space and on the unit polydisc in $\mathbb{C}^n$. The obtained results provide the bounds of Toeplitz determinants for various subclasses of normalized univalent functions in higher dimensions.
\end{abstract}

\vspace{0.5cm}
\maketitle

\section{Introduction}
        Coefficient problems play a significant role in the growth of Geometric function theory. In fact, Bieberbach conjecture, which deals with the bounds of the coefficients of normalized  analytic univalent functions defined on the unit disk, took 68 years to prove. While solving the conjecture, various new concepts were  developed. Koebe one-quarter theorem illustrates an application of the coefficient problem, as it is proved using the second coefficient bound. However, Cartan \cite{Cart} stated that Bieberbach conjecture does not hold in case of several complex variables. There are counterexamples, which show that many results in the Geometric function theory of one complex variable are not applicable for several complex variables (see \cite{Gong}). We use the following notations for functions of one complex variable,

      Let $\mathcal{S}$ be the class of analytic univalent functions in  the unit disk $\mathbb{U}= \{ z \in \mathbb{C}: \vert z \vert <1 \}$, having the series expansion of the form $g(z) = z + \sum_{n=2}^\infty b_n z^n$. Let $\mathcal{S}^*$, $\mathcal{S}^*(\alpha)$ and $\mathcal{SS}^*(\gamma)$ respectively denote the subclasses of $\mathcal{S}$, which contain starlike functions, starlike functions of order $\alpha$ $(0 \leq \alpha <1)$ and strongly starlike functions of order $\gamma$ $(0 < \gamma \leq 1)$. For more details about these classes, we refer \cite{GraKoh}.

      For the class $\mathcal{S}$ and its subclasses, various coefficient problems are studied. In particular, Ali et al. \cite{Ali} obtained the sharp bounds of Toeplitz determinants, when the entries of the Toeplitz matrix are the coefficients of function $g \in \mathcal{S}$ and certain  of its subclasses.

   For the coefficients $\left\{b_k\right\}_{k\geq 2}$ of the function $g(z) = z + \sum_{n=2}^\infty b_n z^n $, the Toeplitz matrix  is defined by
\begin{equation*}
     T_{m,n}(g)= \begin{bmatrix}
	b_n & b_{n+1} & \cdots & b_{n+m-1} \\
	b_{n+1} & b_n & \cdots & b_{n+m-2}\\
	\vdots & \vdots & \vdots & \vdots\\
    b_{n+m-1} & b_{n+m-2} & \cdots & b_n\\
	\end{bmatrix}.
\end{equation*}
    Thus, the second order Toeplitz determinant is
\begin{equation}\label{intilb}
    \det{T_{2,2}(g)}=  b_2^2  -  b_3^2
\end{equation}
   and the third order Toeplitz determinant is given by
\begin{equation}\label{T31}
    \det T_{3,1}(g)=
     \begin{vmatrix}
     1 & b_{2} & b_{3} \\
	 b_{2} & 1 & b_{2}\\
	 b_3 & b_{2} & 1\\
     \end{vmatrix}
      = 2 b_2^{2} b_3 - 2 b_2^2-  b_3^2+1.
\end{equation}
    Ye and Lim \cite{LHLIM} showed that any $n \times n$ matrix over $\mathbb{C}$ generically can be written as the product of some Toeplitz matrices or Hankel matrices. Toeplitz matrices and Toeplitz determinants have numerous applications in the field of pure as well as applied mathematics. They arise in partial differential equations, algebra, signal processing and time series analysis. For more applications of Toeplitz matrices and Toeplitz determinants, we refer \cite{LHLIM} and the references cited therein.

  In 2017, Ali et al. \cite{Ali} determined the bound of $\vert \det T_{2,2}(g) \vert$ and $\vert \det T_{3,1}(g)\vert$, when entries of $T_{m,n}(g)$ are the coefficients of starlike functions. Recently, Ahuja et al. \cite{Ahuja} obtained the bounds for various subclasses of starlike functions.
\begin{theorem A}\label{thmA}\cite{Ali}
   If $g \in \mathcal{S}^*$, then the following sharp bounds hold:
   $$ \vert \det T_{2,2}(g) \vert \leq 13 \;\;  \text{and} \;\; \vert \det T_{3,1}(g) \vert \leq  24.$$
\end{theorem A}
\begin{theorem B}\label{thmB}\cite{Ahuja}
   If $ g \in \mathcal{S}^*(\alpha)$, then
   $$ \vert \det T_{2,2}(g) \vert \leq (1 - \alpha)^2 ( 4 \alpha^2 - 12 \alpha + 13). $$
   For $\alpha \in [0,2/3]$, the following inequality holds:
   $$ \vert \det T_{3,1}(g) \vert \leq 12 \alpha^4 -52 \alpha^3 + 91 \alpha^2 -74 \alpha + 24.$$
   All these estimations are sharp.
\end{theorem B}
\begin{theorem C}\cite{Ahuja} Let $g \in \mathcal{S}^*[D,E]$, then the following sharp estimations hold:
\begin{enumerate}[(i)]
  \item If $\vert D - 2 E \vert \geq 1$, then
    $$ \vert \det T_{2,2}(g) \vert \leq \frac{(D -E )^2 ( D^2 + 4 E^2 - 4 D E + 4)}{4}.    $$
  \item  If $E \leq \min \{ (D -1)/2, (3 D -1)/2 \}$, then
    $$ \vert \det T_{3,1}(g) \vert \leq 1 + 2 ( D - E)^2 + (3 D^2 - 5 D E +  2 E^2) (D^2 - 3 D E + 2 E^2)/4.$$
\end{enumerate}
\end{theorem C}
    The class $\mathcal{S}^*[D,E]$ denotes the class of Janowski starlike functions \cite{Janow}, where $-1 \leq E < D \leq 1$.

    Coefficient problems are also studied in the case of several variables.  For instance, Xu and Liu \cite{XuLiu} solved the Fekete-Szeg\"{o} problem for a subclass of normalized starlike mappings on the unit ball of a complex Banach space. Xu et al. \cite{Xuetal} obtained the bound of the same for a subclass of normalized quasi-convex mappings of type $B$ on the unit ball of complex Banach space.
    Generalizing this work, Hamada et al. \cite{Ham,Xuetal} also determined the bound of Feketo-Szeg\"{o} type inequality.
     Contrary to Feketo-Szeg\"{o} inequality for various subclasses of $\mathcal{S}$, very few results are known for the inequalities of homogeneous expansions for subclasses of biholomorphic mappings in several complex variables \cite{Grah,Ham2,Ham3,Xu22}.
     Numerous best-possible results concerning the coefficient estimates for subclasses of holomorphic mappings in higher dimensions are obtained  in \cite{Bracc,Grah2, Kohr,Liu,SGIRI,XuLiu2}.

   In higher dimensions, the following notations are used throughout the paper. Let $X$ be a complex Banach space with respect to a norm $\| \cdot\|$ and $\mathbb{C}^n$ denote the space of $n$ complex variables $z=(z_1, z_2, \cdots z_n)'$. Also, let $\mathbb{B} = \{ z \in X: \| z \| < 1\}$ be the unit ball in $X$ and $\mathbb{U}^n$ be the Euclidean unit ball in $\mathbb{C}^n$. Further, the boundary and distinguished boundary of $\mathbb{U}^n$ is denoted by $\partial \mathbb{U}^n$ and $\partial_0 \mathbb{U}^n$ respectively.

    For each $z \in X \setminus\{0 \}$, consider the set
    $$ T_z = \{ l_z \in L(X,\mathbb{C}) : l_z(z) = \| z\|, \| l_z \| = 1\},$$
    where $L(X, Y)$ denotes the set of continuous linear operators from $X$ into a complex Banach space $Y$. Let $I$ denote the identity in $L(X, X)$. This set is non-empty according to the Hahn-Banach theorem.

    Let $\mathcal{H}(\Omega)$ denote the set of all holomorphic mappings from $\Omega$ into $X$. If $g \in \mathcal{H}(\mathbb{B})$, then for each $k =1, 2, \cdots$, there is a bounded symmetric linear mapping
    $$D^k g(z) : \prod_{j=1}^k X \rightarrow X, $$
    called the $k^{th}$ order Fr\'{e}chet derivative of $g$ at $z$ such that
    $$ g(w) = \sum_{k=0}^\infty \frac{1}{k!} D^k g(z) ((w -z)^k )$$
    for all $w$ in some neighborhood of $z$. It is understood that $$ D^0 g(z) ((w -z)^0 )= g(z)$$ and for $k \geq 1$,
    $$ D^k g(z)( (w -z)^k) = D^k g(z) \underbrace{( w-z, w-z, \cdots, w-z) }_\text{ k -times}.$$
    We say that $g$ is normalized if $g(0)=0$ and $D g(0) =I$.
    A holomorphic mapping $g : \Omega \rightarrow X$ is said to be biholomorphic on the domain $\Omega$ if $g(\Omega)$ is a domain in $X$ and the inverse $g^{-1}$ exists and is holomorphic on $g(\Omega)$. If for each $z \in \Omega$, $Dg(z)$ has a bounded inverse, the mapping $g$ is said to be locally biholomorphic.
     As in the finite dimensional case, let $\mathcal{S}(\mathbb{B})$ denote the set of normalized biholomorphic mappings from $\mathbb{B}$ into $X$. The class $\mathcal{S}(\mathbb{B})$ of normalized biholomorphic mappings is not a normal family on the unit ball in $\mathbb{C}^n$.

    On a bounded circular domain $\Omega \subset \mathbb{C}^n$, the first and the $m^{th}$ Fr\'{e}chet derivative of a holomorphic mapping $g : \Omega \rightarrow X$  are written by
    $ D g(z)$ and $D^m g(z) (a^{m-1},\cdot)$ respectively. The matrix representations are
\begin{align*}
    D g(z) &= \bigg(\frac{\partial g_j}{\partial z_k} \bigg)_{1 \leq j, k \leq n}, \\
    D^m g(z)(a^{m-1}, \cdot) &= \bigg( \sum_{p_1,p_2, \cdots, p_{m-1}=1}^n  \frac{ \partial^m g_j (z)}{\partial z_k \partial z_{p_1} \cdots \partial z_{p_{m-1}}} a_{p_1} \cdots a_{p_{m-1}}   \bigg)_{1 \leq j,k \leq n},
\end{align*}
   where $g(z) = (g_1(z), g_2(z), \cdots g_n(z))', a= (a_1, a_2, \cdots a_n)'\in \mathbb{C}^n.$
\\

    The following class was defined by Hamada et al. \cite{Ham4}.
\begin{definition}
   Let $g: \mathbb{B} \rightarrow X$ be a normalized locally biholomorphic mapping and $\alpha \in (0,1)$. We say that $g$ is starlike of order $\alpha$ if
   $$ \left\vert \frac{1}{ \| z\|} l_z ([D g(z)]^{-1} g(z)) - \frac{1}{2 \alpha} \right\vert < \frac{1}{ 2 \alpha}, \quad \forall z \in \mathbb{B}\setminus\{0\}, \; l_z \in T(z) . $$
   In case of $X = \mathbb{C}^n$ and $\mathbb{B}= \mathbb{U}^n$, the above condition is equivalent to
   $$ \left\vert \frac{q_k(z)}{z_k} - \frac{1}{2 \alpha} \right\vert < \frac{1}{2\alpha}, \quad \forall z \in \mathbb{U}^n \setminus \{0 \},$$
   where $$q(z) = ( q_1(z), q_2(z), \cdots , q_n(z))' = (D (g(z)))^{-1}g(z)$$
   is a column vector in $\mathbb{C}^n$ and $k$ satisfies $$\vert z_k \vert = \| z \| = \max_{1 \leq j \leq n} \{ \vert z_j \vert \}.$$
   For $\mathbb{B} = \mathbb{U}$ and $X = \mathbb{C}$,  the relation is equivalent to
   $$ \RE \frac{ z g'(z)}{g(z)} > \alpha, \quad z \in \mathbb{U}. $$
  Let $\mathcal{S}^*_\alpha (\mathbb{B})$ denote the class of starlike mappings of order $\alpha$ on $\mathbb{B}$. When $X = \mathbb{C}$ and $\mathbb{B} = \mathbb{U}$, the class $\mathcal{S}^*_{\alpha}(\mathbb{U})$ is denoted by $\mathcal{S}^*(\alpha).$
\end{definition}
\begin{definition}\cite{Kohr2}
   Let $g: \mathbb{B} \rightarrow X$ be a normalized locally biholomorphic mapping and $\gamma \in (0,1]$. We say that $f$ is strongly starlike mapping of order $\gamma$ if
   $$ \left\vert \arg l_z ([D g(z)]^{-1} g(z))  \right\vert < \frac{\pi}{ 2} \gamma, \quad \forall z \in \mathbb{B}\setminus\{0\}, \; l_z \in T(z) . $$
   In case of $\mathbb{B}= \mathbb{U}^n$ and $X = \mathbb{C}^n$, the above condition is equivalent to
   $$ \left\vert \arg \frac{q_j(z)}{z_j}  \right\vert < \frac{\pi}{ 2} \gamma, \quad z \in \mathbb{U}^n\setminus\{0\} . $$
   where $$q(z) = ( q_1(z), q_2(z), \cdots , q_n(z))' = (D (g(z)))^{-1}g(z)$$ is a column vector in $\mathbb{C}^n$ and $j$ satisfies $$\vert z_j \vert = \| z \| = \max_{1 \leq k \leq n} \{ \vert z_k \vert\}.$$
   In case of  $\mathbb{B} = \mathbb{U}$ and $X = \mathbb{C}$, the relation is equivalent to
   $$ \left\vert \arg \frac{ z g'( z)}{g( z )} \right\vert  < \frac{\pi}{2} \gamma, \quad z \in \mathbb{U}. $$
  Let $\mathcal{SS}^*_\gamma (\mathbb{B})$ denote the class of starlike mappings of order $\gamma$ on $\mathbb{B}$. When $X = \mathbb{C}$ and $\mathbb{B} = \mathbb{U}$, the class $\mathcal{SS}^*_{\gamma}(\mathbb{U})$ is denoted by $\mathcal{SS}^*(\gamma).$
\end{definition}
   For a biholomorphic function $\Phi : \mathbb{U} \rightarrow \mathbb{C}$, which satisfies $\Phi(0)=1$ and $\RE \Phi(z) >0$,  Kohr \cite{Kohr} introduced the class $\mathcal{M}_\Phi$ containing the functions $p \in \mathcal{H}(\mathbb{B})$ such that $D p(0) =I$ and $\| z\|/l_z( p(z)) \in \Phi (\mathbb{U}).$ Here, we additionally take $\Phi'(0) >0$, $\Phi''(0) \in \mathbb{R}$ and define the following:
\begin{definition}\label{defn1}
   Let $\Phi : \mathbb{U} \rightarrow \mathbb{C}$ be a biholomorphic function such that $\Phi(0) =1$, $\RE \Phi(z) > 0$, $\Phi'(0)>0$ and $\Phi''(0) \in \mathbb{R}$. We define $\mathcal{M}_\Phi$ to be the class of mappings given by
   $$\mathcal{M}_\Phi = \left\{ p \in \mathcal{H}(\mathbb{B}) : p(0) =0, D(p(0))=I, \frac{\| z\|}{l_z ( p(z) )} \in \Phi(\mathbb{U}), z\in \mathbb{B}\setminus \{ 0 \},\; l_z \in T(z)  \right\}.$$
   For $\mathbb{B} = \mathbb{U}^n$ and $X = \mathbb{C}^n$, the above relation is equivalent to
   $$\mathcal{M}_\Phi = \left\{ p \in \mathcal{H}(\mathbb{U}^n) : p(0) =0, D(p(0))=I, \frac{z_k}{ p_k(z) } \in \Phi(\mathbb{U}), z\in \mathbb{U}^n\setminus \{ 0 \}  \right\},$$
   where $p(z) = ( p_1(z), p_2(z), \cdots , p_n(z))' $ is a column vector in $\mathbb{C}^n$ and $k$ satisfies
   $$\vert z_k \vert = \| z \| = \max_{1 \leq j \leq n} \{ \vert z_j \vert\}.$$
   For $\mathbb{B} = \mathbb{U}$ and $X = \mathbb{C}$,  the relation is equivalent to
    $$\mathcal{M}_\Phi = \left\{ p \in \mathcal{H}(\mathbb{U}) : p(0) =0, p'(0)=1, \frac{z}{ p(z) } \in \Phi(\mathbb{U}), z \in \mathbb{U} \right\}.$$
\end{definition}
   Also, note that, if $g\in \mathcal{H}(\mathbb{B})$ and $D (g(z))^{-1} g(z) \in \mathcal{M}_\Phi$, then suitable choices of $\Phi$  in Definition \ref{defn1} provide different subclasses of holomorphic mappings. For instance, when $\Phi(z) = (1+z)/(1 -z)$, $\Phi(z) = (1 + (1-2 \alpha)z)/(1-z)$ and $\Phi(z) = ((1+z)/(1-z))^\gamma$, we easily obtain that  $g \in \mathcal{S}^*(\mathbb{B})$ , $g \in  \mathcal{S}^*_\alpha(\mathbb{B})$ and $g \in  \mathcal{SS}^*_\gamma(\mathbb{B})$ respectively.

    In this paper, we obtain the sharp bounds of $ \vert \det T_{2,2}(g) \vert$ and $\vert \det T_{3,1}(g) \vert$ for a class of holomorphic functions in the unit disk, which contain the above results as special cases. Further, these results are generalized in higher dimensions.
\section{Main Results}
    First, we find the bounds of  $ \vert \det T_{2,2}(g) \vert$ and $\vert \det T_{3,1}(g) \vert$ for a class of holomorphic mappings defined on $\mathbb{U}.$
\begin{theorem}\label{thm1}
    Let $g(z) = z + b_2 z^2 +b_3 z^3 + \cdots \in  \Phi(\mathbb{U})$, where $\Phi$ is same as given in Definition \ref{defn1} and satisfy
     $$ \vert \Phi''(0) + 2 (\Phi'(0))^2 \vert \geq 2 \Phi'(0) >0 .$$
     If ${g(z)}/{g'(z)} \in \mathcal{M}_\Phi$, then
    $$ \lvert T_{2,2}(g) \rvert \leq  \frac{(\Phi'(0))^2}{4}\left( \frac{1}{2} \frac{\Phi''(0)}{\Phi'(0)} +  \Phi'(0) \right)^2 + ( \Phi'(0))^2.$$
    The bound is sharp.
\end{theorem}
\begin{proof}
   Since $g(z)/g'(z) \in \mathcal{M}_\Phi$, therefore we have
   $$ G(z) :=\frac{z g'(z)}{g(z)} \in \Phi(\mathbb{U}),$$
   which yields $G \prec \Phi$. For $g(z) = z + b_2 z^2 +b_3 z^3 + \cdots $, the Taylor series expansion of $G(z)$ is given by
   $$ G(z) = 1 + b_2 z + (2 b_3 - b_2^2) z^2 + \cdots.$$
    Xu et al. \cite{Xu} proved that
\begin{equation}\label{FS}
    \vert b_3 - \lambda b_2^2 \vert \leq \frac{\vert \Phi'(0) \vert}{2} \max \left\{ 1,  \left\lvert \frac{1}{2} \frac{\Phi''(0)}{\Phi'(0)} + (1 - 2 \lambda)  \Phi'(0) \right\rvert \right\} , \quad \lambda \in \mathbb{C}.
\end{equation}
    Thus, whenever $  \vert \Phi''(0) + 2 (\Phi'(0))^2 \vert  \geq 2  \Phi'(0) $, the equation (\ref{FS}) yields
\begin{equation}\label{a3}
   \lvert b_3 \rvert \leq \frac{\Phi'(0)}{2}  \left\vert \frac{1}{2} \frac{\Phi''(0)}{\Phi'(0)} +  \Phi'(0) \right\vert.
\end{equation}
  Further, using the bound $\lvert G'(0) \rvert \leq  \Phi'(0) $, we obtain
\begin{equation}\label{a2}
    \vert b_2 \vert \leq  \Phi'(0) .
\end{equation}
   From (\ref{intilb}), we have
\begin{align*}
    \vert \det T_{2,2}(g) \vert  &= \lvert b_3^2 - b_2^2 \rvert \leq \lvert b_3 \rvert^2 + \lvert b_2 \rvert^2 .
\end{align*}
    Clearly, the required bound follows directly from the above relation together with the bounds of  $\vert b_3 \vert$ and $\vert b_2\vert$ given in (\ref{a3}) and (\ref{a2}) respectively.

    To see the sharpness of the bound, consider the function $ g_\Phi : \mathbb{U} \rightarrow \mathbb{C}$ given by
\begin{equation}\label{tildf}
    g_\Phi(z) = z \exp \int_0^z \frac{(\Phi(i t)-1)}{t}dt = 1 + i \Phi'(0) z- \frac{1}{2} \bigg( (\Phi'(0))^2  + \frac{\Phi''(0)}{2} \bigg) z^2 + \cdots.
\end{equation}
   It can be easily seen that $g_\Phi(z)/g_\Phi'(z)  \in \mathcal{M}_\Phi $ and
  $$ \vert \det T_{2,2}(g_\Phi) \vert  =  \frac{1}{4} \bigg( (\Phi'(0))^2  + \frac{\Phi''(0)}{2} \bigg)^2 +  (\Phi'(0))^2,$$
  which shows that the bound is sharp and completes the proof.
\end{proof}
\begin{theorem}\label{thm2}
   Let $g(z) = z + b_2 z^2 +b_3 z^3 + \cdots \in \Phi(\mathbb{U})$, where $\Phi$ is same as given in Definition \ref{defn1} and satisfy
      $$ 2 \Phi'(0) - 2 (\Phi'(0))^2 \leq {\Phi''(0)} \leq 6 (\Phi'(0))^2 - 2 \Phi'(0)   . $$
      If ${g (z)}/{g'(z)} \in \mathcal{M}_\Phi$, then
    $$  \vert \det T_{3,1}(g) \vert \leq 1 + 2 ( \Phi'(0) )^2 + \frac{ (\Phi'(0)) ^2}{4} \bigg(  3 \Phi'(0) - \frac{\Phi''(0)}{2 \Phi'(0)} \bigg)  \bigg( \frac{\Phi''(0)}{2 \Phi'(0)} + \Phi'(0) \bigg).$$
    The bound is sharp.
\end{theorem}
\begin{proof}
    Since $\  \Phi''(0) + 2 ( \Phi'(0))^2  \geq 2 {\Phi'(0)} ,$ by (\ref{a3}),  we obtain
\begin{equation}\label{a32}
     \vert b_3 \vert \leq \frac{\Phi'(0)}{2}  \bigg( \frac{1}{2} \frac{\Phi''(0)}{\Phi'(0)} +  \Phi'(0) \bigg).
\end{equation}    Also, $ 6 (\Phi'(0))^2 - \Phi''(0) \geq 2  \Phi'(0) $ holds, hence (\ref{FS}) gives
\begin{equation}\label{FSP}
     \vert b_3 - 2 b_2^2 \vert \leq \frac{ \Phi'(0) }{2}   \bigg(   3  \Phi'(0)  - \frac{1}{2} \frac{\Phi''(0)}{\Phi'(0)} \bigg) .
\end{equation}
   From (\ref{T31}), we have
\begin{align*}
   \lvert \det T_{3,1}(g) \rvert &= \vert 2 b_2^{2} b_3 - 2 b_2^2-  b_3^2+1 \vert \\
                 &\leq  1 + 2 \vert b_2 \vert^2 + \vert b_3\vert \vert b_3 - 2 b_2^2\vert.
\end{align*}
   Using the estimates for the second and third coefficients given in (\ref{a32}) and (\ref{a3}) together with the bound of $\vert b_3 - 2 b_2^2 \vert$ given in (\ref{FSP}), required bound follows.

   The estimate is sharp for the function $g_\Phi(z) = z + \sum_{n=2}^\infty b_n z^n$ given by (\ref{tildf}). For this function, we have
   $$ 1 - 2 b_2^2 - b_3 (b_3 - 2 b_2^2) =  1 + 2 ( \Phi'(0))^2 + \frac{( \Phi'(0) )^2}{4} \bigg(  3 \Phi'(0) - \frac{\Phi''(0)}{2 \Phi'(0)}  \bigg)  \bigg( \frac{\Phi''(0)}{2 \Phi'(0)} + \Phi'(0) \bigg),$$
   which proves the sharpness of the bound.
\end{proof}
\begin{remark}
     By taking $\Phi(z)=(1 + z)/(1 - z)$, $\Phi(z)=(1 + (1- 2 \alpha) z)/(1 - z)$ and $\Phi(z) =(1 + D z)/(1 + E z)$, Theorem \ref{thm1} and \ref{thm2} can be deduced to Theorem A, Theorem B and Theorem C, respectively.
\end{remark}

    The bounds for other classes can also be obtained by changing the corresponding function $\Phi$. For $\Phi(z)= ((1+z)/(1-z))^\gamma$, the following result follows for the class $\mathcal{SS}^*(\gamma).$
\begin{corollary}
    If $g \in \mathcal{SS}^*(\gamma)$, then for $\gamma \in [1/3,1]$, the followings sharp inequalities hold:
    $$ \vert \det T_{2,2}(g) \vert \leq 9 \gamma^4 + 4 \gamma^2 \;\; \text{and} \;\;  \vert \det T_{3,1}(g) \vert \leq  15 \gamma^4 + 8 \gamma^2 + 1.$$
\end{corollary}
      Next, we extend the above results on the unit ball $\mathbb{B}$ and on the unit polydisc $\mathbb{U}^n.$
\begin{theorem}\label{thmB1}
   Let  $g \in \mathcal{H}(\mathbb{B}, \mathbb{C})$ with $g(0)=1$ and suppose that $G(z) =  z g(z) $. If $(D G(z))^{-1} G(z) \in \mathcal{M}_\Phi$ such that $\Phi$ satisfy
   $$\lvert \Phi''(0) + 2 (\Phi'(0))^2 \rvert \geq 2  \Phi'(0), $$
   then
    $$ \bigg\vert \bigg( \frac{ l_z (D^2 G(0) (z^2))}{2! \vert\vert z \vert\vert^2} \bigg)^2 - \bigg(\frac{ l_z (D^3 G(0) (z^3))}{3! \vert\vert z \vert\vert^3} \bigg)^2 \bigg\vert \leq \frac{(\Phi'(0))^2}{4}\bigg( \frac{1}{2} \frac{\Phi''(0)}{\Phi'(0)} +  \Phi'(0) \bigg)^2 + ( \Phi'(0) )^2  .$$
   The bound is sharp.
\end{theorem}
\begin{proof}
  Xu et al. \cite[Theorem 3.2]{Xu} proved that
\begin{equation}\label{FSB}
\begin{aligned}
   \bigg\vert & \frac{ l_z (D^3 G(0) (z^3))}{3! \vert\vert z \vert\vert^3} -  \lambda \bigg(\frac{ l_z (D^2 G(0) (z^2))}{2! \vert\vert z \vert\vert^2} \bigg)^2 \bigg\vert \\
   &\leq \frac{\vert \Phi'(0) \vert}{2} \max \left\{ 1,  \left\lvert \frac{1}{2} \frac{\Phi''(0)}{\Phi'(0)} + (1 - 2 \lambda)  \Phi'(0) \right\rvert \right\} , \quad \lambda \in \mathbb{C}, z \in \mathbb{B}\setminus \{0 \}.
\end{aligned}
\end{equation}
   Since $\lvert \Phi''(0) + 2 (\Phi'(0))^2 \rvert \geq 2  \Phi'(0),$ the above inequality gives
\begin{equation}\label{a3B}
     \bigg\vert  \frac{ l_z (D^3 G(0) (z^3))}{3! \vert\vert z \vert\vert^3} \bigg\vert \leq \frac{\Phi'(0)}{2}  \left\vert \frac{1}{2} \frac{\Phi''(0)}{\Phi'(0)} +  \Phi'(0) \right\vert.
\end{equation}
    On the other hand, applying a similar method as in \cite[Theorem 7.1.14]{GraKoh} (also see \cite[Theorem 3.2]{Xu}), we obtain
    $$ (D G(z))^{-1} = \frac{1}{g(z)} \bigg( I - \frac{\frac{z D g(z)}{g(z)}}{1 + \frac{D g(z) z}{g(z)}} \bigg). $$
    Therefore
    $$ (D G(z))^{-1} G(z) = z \bigg( \frac{z g(z) }{g(z) + D g(z) z} \bigg),  \quad z\in \mathbb{B}, $$
    which directly gives
\begin{equation}\label{newe}
   \frac{\| z\|}{l_z ((D G(z))^{-1} G(z))} = 1 + \frac{D g(z) z}{g(z)}.
\end{equation}
  For fix $z\in X\setminus \{ 0 \}$ and $z_0 = \frac{z}{\|z \|}$,  define the function $ h : \mathbb{U} \rightarrow \mathbb{C}$ such that
\begin{equation*}
    h(\zeta) = \left\{ \begin{array}{ll}
     \dfrac{\zeta}{ l_z ((D G(\zeta z_0))^{-1} G( \zeta z_0) )}, & \zeta \neq 0, \\ \\
    1, & \zeta =0.
    \end{array}
    \right.
\end{equation*}
   Then $h \in \mathcal{H}(\mathbb{U})$ and $h(0)=1  = \Phi(0) $. Further, since $(D G(z))^{-1} G(z) \in \mathcal{M}_\Phi$, we find that
\begin{align*}
   h(\zeta) =& \frac{\zeta}{l_z ((D G(\zeta z_0))^{-1} G( \zeta z_0) ) } = \frac{\zeta}{l_{z_0} ((D G(\zeta z_0))^{-1} G( \zeta z_0) ) } \\
             =& \frac{\| \zeta z_0 \| }{l_{ \zeta z_0} ((D G(\zeta z_0))^{-1} G( \zeta z_0) ) } \in \Phi(\mathbb{U}), \quad  \zeta \in \mathbb{U}.
\end{align*}
   Taking (\ref{newe}) into consideration, we obtain
\begin{equation}\label{accr}
    h(\zeta) = \frac{\| \zeta z_0 \| }{l_{ \zeta z_0} ((Dg(\zeta z_0))^{-1} g( \zeta z_0) ) }  =  1 + \frac{D g(\zeta z_0)\zeta z_0}{g(\zeta z_0)}.
\end{equation}
   In view of the Taylor series expansions of $h(\zeta)$ and $g(\zeta z_0)$, the above equation gives
\begin{align*}
   \bigg(1 + & h'(0) \zeta  + \frac{h''(0)}{2} \zeta^2 + \cdots \bigg)\bigg( 1 + Dg(0)(z_0) \zeta + \frac{ D^2 g(0)(z_{0}^2)}{2} \zeta^2 + \cdots \bigg)  \\
   & =\bigg( 1 + Dg(0)(z_0) \zeta + \frac{ D^2 g(0)(z_{0}^2)}{2} \zeta^2 + \cdots \bigg) \bigg( Dg(0)(z_0) \zeta +  D^2 g(0)(z_{0}^2)\zeta^2 + \cdots \bigg).
\end{align*}
   Comparison of homogeneous expansions yield that $ h'(0) = D g(0)(z_0).$ That is
\begin{equation}\label{new2}
    h'(0) \| z \| = D g(0)(z).
\end{equation}
   Since $G(z) = z g(z)$, therefore, we have
\begin{equation}\label{mor}
    \frac{ D^2 G(0) (z^2)}{2! } =  D g(0)(z)  z.
\end{equation}
   Moreover, from (\ref{mor}), we conclude that
\begin{align}\label{alig}
    \frac{ l_z (D^2 G(0) (z^2))}{2! } &=  D g(0)(z) \| z \|.
\end{align}
    Thus, equation (\ref{alig}) together with (\ref{new2}) gives
\begin{align*}
   \bigg\vert \frac{ l_z (D^2 G(0) (z^2))}{2! }\bigg\vert = \vert D g(0)(z) \| z \| \vert =\vert h'(0) \| z \|^2 \vert .
\end{align*}
   Since $h \prec \Phi$, therefore $\vert h'(0) \vert \leq \Phi'(0) $. Consequently, we obtain
\begin{equation}\label{a2B}
    \bigg\vert \frac{ l_z (D^2 G(0) (z^2))}{2! \| z \|^2}\bigg\vert \leq  \Phi'(0) .
\end{equation}
   Using the bounds given in (\ref{a2B}) and (\ref{a3B}) together with
\begin{align*}
    \bigg\lvert \bigg( \frac{ l_z (D^2 G(0) (z^2))}{2! \| z \|^2}  \bigg)^2 - \bigg(  \frac{ l_z (D^3 G(0) (z^3))}{3! \vert\vert z \vert\vert^3} \bigg)^2 \bigg\rvert  \leq  \bigg\vert  \frac{ l_z (D^3 G(0) (z^3))}{3! \vert\vert z \vert\vert^3} \bigg\vert^2 + \bigg\vert \frac{ l_z (D^2 G(0) (z^2))}{2! \| z \|^2}  \bigg\vert^2
\end{align*}
   the required bound follows.

   To see the sharpness, consider the function $G$ given by
\begin{equation}\label{extB}
    G(z) = z \exp \int_0^{l_u(z)} \frac{( \Phi(i t)-1) }{t}dt, \quad z\in \mathbb{B}, \quad \vert\vert u \vert\vert=1.
\end{equation}
   It is a simple exercise to see that $ (D  G(z))^{-1}  G(z) \in \mathcal{M}_\Phi$ and a quick calculation reveals that
   $$  \frac{D^2  G(0) (z^2)}{2!}= i \Phi'(0) l_u(z) z \;\; \text{and} \;\; \frac{D^3  G(0) (z^3)}{3!} = - \frac{1}{2} \left( \frac{\Phi''(0)}{2} + (\Phi'(0))^2 \right) (l_u (z))^2 z. $$
   In view of the above equations, we have
   $$ \frac{l_z (D^2  G(0) ( z^2) )}{2!} = i \Phi'(0) l_u(z) \| z \|  $$
   and
   $$  \frac{l_z (D^3  G(0) ( z^3 )) \| z \|}{3!} = - \frac{1}{2} \left( \frac{\Phi''(0)}{2} + (\Phi'(0))^2 \right) (l_u (z))^2 \| z \|^2. $$
   Setting $z = r u$ $(0< r <1)$, we get
\begin{equation}\label{cftB}
   \frac{l_z (D^2  G(0) ( z^2) )}{2! \|z \|^2} = i \Phi'(0) \;\; \text{and} \;\;  \frac{l_z (D^3  G(0) ( z^3) ) }{3! \| z \|^3}  = - \frac{1}{2} \left( \frac{\Phi''(0)}{2} + (\Phi'(0))^2 \right) .
\end{equation}
   Thus for the function $G$, we have
  $$ \bigg\vert \bigg( \frac{ l_z (D^2 G(0) (z^2))}{2! \vert\vert z \vert\vert^2} \bigg)^2 - \bigg(\frac{ l_z (D^3 G(0) (z^3))}{3! \vert\vert z \vert\vert^3} \bigg)^2 \bigg\vert = \frac{(\Phi'(0))^2}{4}\left( \frac{1}{2} \frac{\Phi''(0)}{\Phi'(0)} +  \Phi'(0) \right)^2 + \lvert \Phi'(0)\rvert^2, $$
  which proves the sharpness of the bound.
\end{proof}
\begin{theorem}\label{thmB2}
   Let  $g \in \mathcal{H}(\mathbb{B}, \mathbb{C})$ with $g(0)=1$ and suppose that $G(z) =  z g(z)$. If $(D G(z))^{-1} G(z) \in \mathcal{M}_\Phi$ such that $\Phi$ satisfy
   $$ 2 \Phi'(0) - 2 (\Phi'(0))^2 \leq {\Phi''(0)} \leq 6 (\Phi'(0))^2 - 2 \Phi'(0), $$
    then
   $$ \vert 2 b_2^2 b_3  - b_3^2 - 2 b_2^2 + 1 \vert \leq 1 + 2 ( \Phi'(0))^2 + \frac{ ( \Phi'(0) )^2}{4} \bigg(  3 \Phi'(0) - \frac{\Phi''(0)}{2 \Phi'(0)} \bigg)  \bigg( \frac{\Phi''(0)}{2 \Phi'(0)} + \Phi'(0) \bigg),$$
   where
\begin{align*}
   b_3 = \frac{ l_z (D^3 G(0) (z^3))}{3! \vert\vert z \vert\vert^3} \;\; \text{and} \;\;  b_2 &= \frac{ l_z (D^2 G(0) (z^2))}{2! \vert\vert z \vert\vert^2}.
\end{align*}
   The bound is sharp.
\end{theorem}
\begin{proof}
      Since $ 2 \Phi'(0) <  \Phi''(0) + 2 (\Phi'(0))^2$, therefore from (\ref{FSB}), we have
\begin{equation}\label{2a3B}
    \bigg\vert  \frac{ l_z (D^3 G(0) (z^3))}{3! \vert\vert z \vert\vert^3} \bigg\vert \leq \frac{\Phi'(0)}{2} \bigg( \frac{1}{2} \frac{\Phi''(0)}{\Phi'(0)} +  \Phi'(0) \bigg).
\end{equation}
     Again, since $ 2 \Phi'(0)  + \Phi''(0) \leq 6 (\Phi'(0))^2 ,$ the inequality (\ref{FSB}) directly gives
\begin{equation}\label{FS2B}
     \bigg\vert\frac{ l_z (D^3 G(0) (z^3))}{3! \vert\vert z \vert\vert^3} - 2 \bigg(  \frac{ l_z (D^2 G(0) (z^2))}{2! \vert\vert z \vert\vert^2} \bigg)^2 \bigg\vert \leq  \frac{\Phi'(0) }{2}   \bigg( 3  \Phi'(0)  - \frac{1}{2} \frac{\Phi''(0)}{\Phi'(0)} \bigg).
\end{equation}
     Also, we have
\begin{equation}\label{T31B}
    \vert 2 b_2^2 b_3  - b_3^2 - 2 b_2^2 + 1 \vert \leq  1 + 2 \vert b_2 \vert^2 + \vert b_3\vert \vert b_3 - 2 b_2^2\vert.
\end{equation}
  The required bound is derived by using the estimates given in (\ref{a2B}) and (\ref{2a3B}), and the bound given by (\ref{FS2B}) in the above inequality.

   The equality case holds for the function $G(z)$ defined by (\ref{extB}). It follows from (\ref{cftB}) that for this function, we have  $b_2 = i \Phi'(0)$,
   $ b_3  = -( \Phi''(0) + 2(\Phi'(0))^2 )/4$ and hence
   $$   1 - b_3 ( b_3 - 2 b_2^2 ) - 2 b_2^2  = 1 + 2 ( \Phi'(0) )^2 + \frac{ ( \Phi'(0) )^2}{4} \bigg( 3 \Phi'(0)  - \frac{\Phi''(0)}{2 \Phi'(0)} \bigg)  \bigg( \frac{\Phi''(0)}{2 \Phi'(0)} + \Phi'(0) \bigg),$$
   which shows the sharpness of the bound.
\end{proof}
\begin{theorem}\label{ThmUn1}
    Let $g \in \mathcal{H}(\mathbb{U}^n, \mathbb{C})$ with $g(0)=1$  and suppose that $G(z) =  z g(z)$. If $(D G(z))^{-1} G(z) \in \mathcal{M}_\Phi$ such that $\Phi$ satisfies
    $$ \lvert \Phi''(0) + 2 (\Phi'(0))^2 \rvert \geq 2  \Phi'(0),$$
   then
\begin{equation}\label{res}
\begin{aligned}
\left.
\begin{array}{ll}
   \bigg\|& \bigg(  \dfrac{ D^3 G(0) (z^3)}{3! } \bigg)^2  - \bigg( \dfrac{D^2 G(0) (z^2)}{2! } \bigg)^2 \bigg\|  \\
      &\leq   \dfrac{( \Phi'(0) )^2 \| z\|^6}{4} \bigg(  \dfrac{1}{2} \dfrac{\Phi''(0)}{\Phi'(0)} +   \Phi'(0) \bigg)^2  +\left( \vert \Phi'(0)\vert \| z \|^2 \right)^2 , \;\;\; z \in \mathbb{U}^n.
\end{array}
\right\}
\end{aligned}
\end{equation}
   The bound is sharp.
\end{theorem}
\begin{proof}
   For $z \in \mathbb{U}^n \setminus \{ 0 \}$, let $z_0 = \frac{z}{\| z\|}$. Define $h_k : \mathbb{U} \rightarrow \mathbb{C}$ such that
\begin{equation}\label{hkzeta}
   h_k (\zeta) =
\left\{
\begin{array}{ll}
     \dfrac{\zeta z_k}{p_k (\zeta z_0) \| z_0 \|}, & \zeta \neq 0,\\
     1 , & \zeta =0,
\end{array}
\right.
\end{equation}
   where $p(z) = (D G(z))^{-1} G(z)$ and $k$ satisfies $\vert z_k \vert = \| z\| = \max_{1 \leq j \leq n} \{ \vert z_j \vert \}$.
   Since $(D (G(z)))^{-1} G(z) \in \mathcal{M}_\Phi$, we have $h_k (\zeta) \in \Phi (\mathbb{U})$.
   Further, using (\ref{accr}), we have
   $$ h_k (\zeta) =  1 + \frac{D g( \zeta z_0) \zeta z_0}{g(\zeta z_0)} $$
   or equivalently,
   $$  h_k (\zeta) g(\zeta z_0) =  g(\zeta z_0) + D g(\zeta z_0) \zeta z_0 . $$
   A comparison of homogeneous expansions obtained by the Taylor series expansions of $g$ and $h_k$ about $\zeta$ gives
\begin{equation}\label{use0}
    h'_k (0) = D g(0)(z_0), \quad \frac{h''_k (0)}{2} = D^2 g(0) (z_0^2 ) - (D g(0) (z_0))^2.
\end{equation}
   Also, using $G(z_0) =  z_0 g(z_0)$, we have
\begin{equation}\label{use}
    \frac{D^3 G_k(0) (z_0^3)}{3!} = \frac{D^2 g(0) (z_0^2)}{2!} \frac{z_k}{ \| z\|}\;\;\; \text{and} \;\; \;  \frac{D^2 G_k(0) (z_0^2)}{2!} = D g(0) (z_0) \frac{z_k}{ \| z\|}.
\end{equation}
   Thus, from (\ref{use0}) and (\ref{use}), we obtain
   $$ \frac{D^2 G_k(0) (z_0^2)}{2!} \frac{\| z \|}{z_k}  =   h'_k(0), $$
   which gives
\begin{equation*}
    \left\vert \frac{D^2 G_k(0) (z_0^2)}{2!} \frac{\| z \|}{z_k} \right\vert =  \vert  h'_k(0) \vert \leq  \Phi'(0).
\end{equation*}
   If $z_0 \in \partial_0 \mathbb{U}^n$, then
\begin{equation}\label{a2^2}
    \left\vert \frac{D^2 G_k(0) (z_0^2)}{2!}  \right\vert  \leq  \Phi'(0).
\end{equation}   Since
   $$  \frac{D^2 G_k(0) (z_0^2)}{2!}, \quad k = 1,2,\cdots n$$
   are holomorphic on $\overline{\mathbb{U}}^n$, by virtue of the maximum modulus theorem of holomorphic functions on the unit polydisc, we obtain
   $$ \left\vert \frac{D^2 G_k(0) (z_0^2)}{2! }  \right\vert  \leq  \Phi'(0) , \quad  z_0 \in \partial \mathbb{U}^n, k= 1,2,3,\cdots n.$$
   That is
\begin{equation}\label{use2}
    \left\vert \frac{D^2 G_k(0) (z^2)}{2! }  \right \vert  \leq  \Phi'(0) \| z \|^2, \quad  z \in \mathbb{U}^n, k= 1,2,3,\cdots n .
\end{equation}
   Therefore
    $$ \left\| \frac{D^2 G(0) (z^2)}{2! }  \right\|  \leq   \Phi'(0)  \| z \|^2, \quad  z \in \mathbb{U}^n . $$
    According to the result established by Xu et al.\cite[Theorem 3.3]{Xu}, we have
\begin{equation}\label{FSUn}
\begin{aligned}
\left.
\begin{array}{ll}
    \bigg\vert & \dfrac{ D^3 G_k(0) (z^3)}{3! }  -  \lambda \dfrac{1}{2} D^2  G_k (0)  \bigg( z, \dfrac{D^2 G(0) (z^2)}{2!} \bigg) \bigg\vert \\
    &\leq \dfrac{\vert \Phi'(0) \vert \| z\|^3}{2} \max\bigg\{1,  \left\lvert \dfrac{1}{2} \dfrac{\Phi''(0)}{\Phi'(0)} + (1 - 2 \lambda)  \Phi'(0) \right\rvert \bigg\} , \;\; k= 1,2, \cdots n.
\end{array}
\right\}
\end{aligned}
\end{equation}
   Since $\lvert \Phi''(0) + 2 (\Phi'(0))^2 \rvert \geq 2 \Phi'(0),$ therefore from (\ref{FSUn}), it follows that
\begin{equation}\label{a3Un}
    \bigg\vert  \frac{ D^3 G_k(0) (z^3)}{3! }   \bigg\vert  \leq \frac{\vert \Phi'(0) \vert \| z\|^3}{2}   \left\vert \frac{1}{2} \frac{\Phi''(0)}{\Phi'(0)} +   \Phi'(0) \right\vert  , \quad z \in \mathbb{U}^n ,k = 1,2, \cdots n.
\end{equation}
   Now, using the bounds given in (\ref{use2}) and (\ref{a3Un}), we obtain
\begin{align*}
   \bigg\vert \bigg( & \frac{ D^3 G_k(0) (z^3)}{3! } \bigg)^2  - \bigg( \frac{D^2 G_k(0) (z^2)}{2! } \bigg)^2 \bigg\vert  \\
      &\leq  \frac{ ( \Phi'(0) )^2  \| z\|^6}{4} \bigg( \frac{1}{2} \frac{\Phi''(0)}{\Phi'(0)} +  \Phi'(0)  \bigg)^2  +\left( \Phi'(0) \right)^2 \| z \|^4 , \quad k=1,2, \cdots n.
\end{align*}
    Therefore,
\begin{align*}
   \bigg\| \bigg(  \frac{ D^3 G(0) (z^3)}{3! } \bigg)^2 & - \bigg( \frac{D^2 G(0) (z^2)}{2! } \bigg)^2 \bigg\|  \\
      &\leq  \frac{( \Phi'(0))^2  \| z\|^6}{4} \bigg( \frac{1}{2} \frac{\Phi''(0)}{\Phi'(0)} +   \Phi'(0) \bigg)^2  + (\Phi'(0) )^2  \| z \|^4 ,
\end{align*}
   which is the required bound.

   To prove the sharpness, consider the function
\begin{equation}\label{extUn}
    G(z) = z \exp \int_0^{z_1} \frac{\Phi(i t)-1}{t}dt, \quad z\in \mathbb{U}^n.
\end{equation}
    It is a simple exercise to check that $D(G(z))^{-1}G(z) \in \mathcal{M}_\Phi$. From the above relation, we deduce that
     $$  \frac{D^2  G(0) (z^2)}{2!}= i \Phi'(0) z_1 z  , \;\; \frac{D^3  G(0) (z^3)}{3!} = - \frac{1}{2} \left( \frac{\Phi''(0)}{2} + (\Phi'(0))^2 \right) (z_1)^2 z. $$
   By taking $z= (r, 0, \cdots, 0)$, the equality in (\ref{res}) holds.
\end{proof}
\begin{theorem}\label{thmUn2}
    Let $g \in \mathcal{H}(\mathbb{U}^n, \mathbb{C})$ with $g(0)=1$  and suppose that $G(z) =  z g(z)$.  If $(D G(z))^{-1} G(z) \in \mathcal{M}_\Phi$ such that $\Phi$ satisfy
    $$ 2 \Phi'(0) - 2 (\Phi'(0))^2 \leq {\Phi''(0)} \leq 6 (\Phi'(0))^2 - 2 \Phi'(0) ,$$
    then
\begin{align*}
     \| 2 b_2^2 b_3 & - b_3^2 - 2 b_2^2 + 1 \|   \\
     & \leq 1 +  \frac{( \Phi'(0) )^2 \| z\|^6}{4}   \bigg( 3 \Phi'(0) - \frac{1}{2} \frac{\Phi''(0)}{\Phi'(0)} \bigg)    \bigg( \frac{1}{2} \frac{\Phi''(0)}{\Phi'(0)} +   \Phi'(0) \bigg)  + 2 ( \Phi'(0) )^2 \| z\|^4,
\end{align*}
   where
\begin{align*}
       b_3 = \frac{ D^3 G(0) (z^3)}{3! }   \;\; \text{and} \;\; b_2^2 =  \frac{1}{2} D^2  G (0)  \bigg( z, \frac{D^2 G(0) (z^2)}{2!}\bigg).
\end{align*}
   The bound is sharp.
\end{theorem}
\begin{proof}  Since $G(z) = z g(z)$, we have
    $$  \frac{1}{2} D^2  G_k (0)  \bigg( z_0, \frac{D^2 G(0) (z_0^2)}{2!}\bigg) \frac{ z_k}{\| z \|}  = \bigg( \frac{D^2 G_k(0) (z_0^2)}{2!} \bigg)^2, \quad k=1,2,\cdots n , $$
     where $z_0 = \frac{z_k}{\| z\|}$ and $k$ satisfies $\vert z_k \vert = \| z \| = \max_{ 1 \leq j \leq n}\{ \vert z_j \vert \}$ (see \cite{XuLiu2}).
    If $z_0 \in \partial \mathbb{U}^n$, then
     $$ \bigg\vert  \frac{1}{2} D^2  G_k (0)  \bigg( z_0, \frac{D^2 G(0) (z_0^2)}{2!}\bigg)  \bigg\vert = \bigg\vert \frac{D^2 G_k(0) (z_0^2)}{2!} \bigg\vert^2. $$
    Consider the function $h_k(\zeta)$ defined in (\ref{hkzeta}). Following the same method as in the proof of Theorem \ref{ThmUn1}, we get (\ref{a2^2}), which together with the above relation yield
     $$ \bigg\vert  \frac{1}{2} D^2  G_k (0)  \bigg( z_0, \frac{D^2 G(0) (z_0^2)}{2!}\bigg)  \bigg\vert \leq \vert \Phi'(0) \vert^2. $$
     Since
      $$ \bigg\vert  \frac{1}{2} D^2  G_k (0)  \bigg( z, \frac{D^2 G(0) (z^2)}{2!}\bigg)  \bigg\vert , \quad k=1,2, \cdots n $$
     are holomorphic functions on $\overline{\mathbb{U}}^n$. By virtue of the maximum modulus theorem of holomorphic functions on the unit polydisc, we obtain
\begin{equation}\label{a2^22}
      \bigg\vert  \frac{1}{2} D^2  G_k (0)  \bigg( z, \frac{D^2 G(0) (z^2)}{2!}\bigg)  \bigg\vert \leq \vert \Phi'(0) \vert^2 \| z\|^4 , \quad z \in \mathbb{U}^n , k= 1,2,\cdots n .
\end{equation}
     Furthermore, since $ 2 \Phi'(0) <  \Phi''(0) + 2 (\Phi'(0))^2$, therefore by (\ref{FSUn}), we have
\begin{align}\label{2a3Un}
    \bigg\vert  \frac{ D^3 G_k(0) (z^3)}{3! }   \bigg\vert  \leq \frac{ \Phi'(0)  \| z\|^3}{2} \bigg( \frac{1}{2} \frac{\Phi''(0)}{\Phi'(0)} +   \Phi'(0) \bigg)  , \quad z \in \mathbb{U}^n ,k = 1,2, \cdots n.
\end{align}
     Again, since $\Phi$ satisfies $2 \Phi'(0)  + \Phi''(0) \leq 6 (\Phi'(0))^2,$ the inequality (\ref{FSUn}) gives
\begin{equation}\label{FS2Un}
    \bigg\vert  \frac{ D^3 G_k(0) (z^3)}{3! }  -   D^2  G_k (0)  \bigg( z, \frac{D^2 G(0) (z^2)}{2!} \bigg) \bigg\vert
    \leq \frac{ \Phi'(0)  \| z\|^3}{2}   \bigg( 3  \Phi'(0) - \frac{1}{2} \frac{\Phi''(0)}{\Phi'(0)}  \bigg)
\end{equation}
    for $z \in \mathbb{U}^n$ and $k =1,2 \cdots n.$
   Now, using (\ref{a2^22}), (\ref{2a3Un}) and (\ref{FS2Un}), we have
\begin{align*}
     \bigg\vert & 1 + D^2  G_k (0)  \bigg( z, \frac{D^2 G(0) (z^2)}{2!}\bigg) \bigg(\frac{ D^3 G_k(0) (z^3)}{3! } \bigg) -  D^2  G_k (0)  \bigg( z, \frac{D^2 G(0) (z^2)}{2!}\bigg)\\
      & - \bigg( \frac{ D^3 G_k(0) (z^3)}{3! } \bigg)^2 \bigg\vert  \\
    & \leq  1 + \bigg\vert \frac{ D^3 G_k(0) (z^3)}{3! } \bigg\vert \bigg\vert \frac{ D^3 G_k(0) (z^3)}{3! } - D^2  G_k (0)  \bigg( z, \frac{D^2 G(0) (z^2)}{2!}\bigg) \bigg\vert \\
    &\;\; \;\;\;\;\; + \bigg\vert D^2  G _k(0)  \bigg( z, \frac{D^2 G(0) (z^2)}{2!}\bigg) \bigg\vert \\
    & \leq  1 +  \frac{( \Phi'(0) )^2 \| z\|^6}{4}   \bigg( 3 \Phi'(0) - \frac{1}{2} \frac{\Phi''(0)}{\Phi'(0)} \bigg)    \bigg( \frac{1}{2} \frac{\Phi''(0)}{\Phi'(0)} +   \Phi'(0) \bigg)  + 2 ( \Phi'(0) )^2 \| z\|^4
\end{align*}
    for $ z\in \mathbb{U}^n$ and $k = 1,2 , \cdots n.$
    Therefore
\begin{align*}
     \bigg\| 1 &+ D^2  G (0)  \bigg( z, \frac{D^2 G(0) (z^2)}{2!}\bigg) \bigg(\frac{ D^3 G(0) (z^3)}{3! } \bigg) -  D^2  G (0)  \bigg( z, \frac{D^2 G(0) (z^2)}{2!}\bigg)\\
      & - \bigg( \frac{ D^3 G(0) (z^3)}{3! } \bigg)^2 \bigg\|  \\
    & \leq   1 +  \frac{( \Phi'(0) )^2 \| z\|^6}{4}   \bigg( 3 \Phi'(0) - \frac{1}{2} \frac{\Phi''(0)}{\Phi'(0)} \bigg)    \bigg( \frac{1}{2} \frac{\Phi''(0)}{\Phi'(0)} +   \Phi'(0) \bigg)  + 2 ( \Phi'(0) )^2 \| z\|^4,
\end{align*}
     which is the required bound.

    Using the same argument as in Theorem \ref{ThmUn1}, the function $G$ provided in (\ref{extUn}) leads to the sharpness of the bound, which completes the proof.
\end{proof}
\begin{remark}
\begin{enumerate}[(i)]
  \item When $\mathbb{B} = \mathbb{U}$ and $X = \mathbb{C}$, Theorem \ref{thmB1} and Theorem \ref{ThmUn1} are equivalent to Theorem \ref{thm1}.
  \item In case of $\mathbb{B} = \mathbb{U}$ and $X = \mathbb{C}$, Theorem \ref{thmB2} and \ref{thmUn2} are equivalent to Theorem \ref{thm2}.
\end{enumerate}
\end{remark}
\section{Applications for Various Subclasses}
    If we take $\Phi(z) =(1+z)/(1-z)$, $\Phi(z) =(1+(1- 2 \alpha)z)/(1-z)$ and $\Phi(z) =((1+z)/(1-z))^\gamma$, Theorem \ref{thmB1}-\ref{thmUn2} gives the following bounds (the branch of the power function is taken such that $((1+z)/(1-z))^\gamma$ =1 at $z =0$).
\begin{corollary}
      Let $g\in \mathcal{H}( \mathbb{B}, \mathbb{C})$ with $g(0)=1$ and $G(z) = z g(z) \in \mathcal{S}^*(\mathbb{B})$. Then the following holds:
      $$ \bigg\vert \bigg( \frac{ l_z (D^2 G(0) (z^2))}{2! \vert\vert z \vert\vert^2} \bigg)^2 - \bigg(\frac{ l_z (D^3 G(0) (z^3))}{3! \vert\vert z \vert\vert^3} \bigg)^2 \bigg\vert \leq 13, \quad z\in \mathbb{B}\setminus\{0\}, l_z \in T_z.$$
    If $\mathbb{B} = \mathbb{U}^n$ and $X = \mathbb{C}^n$, then
   $$  \bigg\| \bigg(  \frac{ D^3 G(0) (z^3)}{3! } \bigg)^2  - \bigg( \frac{D^2 G(0) (z^2)}{2! } \bigg)^2 \bigg\|
      \leq 9 {\| z\|^6}  + 4  \| z \|^4  , \; z \in \mathbb{U}^n.   $$
       All the estimates are sharp.
\end{corollary}
\begin{corollary}
      Let $g\in \mathcal{H}( \mathbb{B}, \mathbb{C})$ with $g(0)=1$ and $G(z) = z g(z) \in \mathcal{S}^*_\alpha(\mathbb{B})$. Then the following holds:
      $$ \bigg\vert \bigg( \frac{ l_z (D^2 G(0) (z^2))}{2! \vert\vert z \vert\vert^2} \bigg)^2 - \bigg(\frac{ l_z (D^3 G(0) (z^3))}{3! \vert\vert z \vert\vert^3} \bigg)^2 \bigg\vert \leq  (1 - \alpha)^2 ( 4 \alpha^2 - 12 \alpha + 13), \;\;  z\in \mathbb{B}\setminus\{0\}.$$
    If $\mathbb{B} = \mathbb{U}^n$ and $X = \mathbb{C}^n$, then
   $$  \bigg\| \bigg(  \frac{ D^3 G(0) (z^3)}{3! } \bigg)^2  - \bigg( \frac{D^2 G(0) (z^2)}{2! } \bigg)^2 \bigg\|
      \leq (1 -\alpha)^2 ( (3- 2 \alpha)^2 \| z \|^6  + 4 \| z \|^4), \; z \in \mathbb{U}^n.   $$
       All the estimates are sharp.
\end{corollary}
\begin{corollary}
     Let $g\in \mathcal{H}( \mathbb{B}, \mathbb{C})$ with $g(0)=1$ and $G(z) = z g(z) \in \mathcal{SS}^*_{\gamma}(\mathbb{B})$. Then for $ \gamma \in [1/3,1]$, the following holds:
      $$ \bigg\vert \bigg( \frac{ l_z (D^2 G(0) (z^2))}{2! \vert\vert z \vert\vert^2} \bigg)^2 - \bigg(\frac{ l_z (D^3 G(0) (z^3))}{3! \vert\vert z \vert\vert^3} \bigg)^2 \bigg\vert \leq  9 \gamma^4 + 4 \gamma^2 , \quad  z\in \mathbb{B}\setminus\{0\}, l_z \in T_z.$$
    If $\mathbb{B} = \mathbb{U}^n$ bnd $X = \mathbb{C}^n$, then
   $$  \bigg\| \bigg(  \frac{ D^3 G(0) (z^3)}{3! } \bigg)^2  - \bigg( \frac{D^2 G(0) (z^2)}{2! } \bigg)^2 \bigg\|
      \leq   9 \|z\|^6 \gamma^4  + 4 \|z\|^4 \gamma^2  , \quad z \in \mathbb{U}^n.   $$
       All the estimates are sharp.
\end{corollary}
\begin{corollary}
      Let $g\in \mathcal{H}( \mathbb{B}, \mathbb{C})$ with $g(0)=1$  and $G(z) = z g(z) \in \mathcal{S}^*(\mathbb{B})$. Then the following holds:
      $$   \vert 2 b_2^2 b_3  - b_3^2 - 2 b_2^2 + 1 \vert \leq 24,$$
      where
\begin{align*}
   b_3 = \frac{ l_z (D^3 G(0) (z^3))}{3! \vert\vert z \vert\vert^3}, \;\; \;\;  b_2 &= \frac{ l_z (D^2 G(0) (z^2))}{2! \vert\vert z \vert\vert^2}, \quad l_z \in T_z.
\end{align*}
       The estimate is sharp.
\end{corollary}

\begin{corollary}
      Let $g \in \mathcal{H}(\mathbb{B}, \mathbb{C})$ with $g(0)=1$ and $G(z) = z g(z) \in \mathcal{S}^*_\alpha(\mathbb{B})$. Then for $\alpha \in [0,2/3]$, the following holds:
      $$  \vert 2 b_2^2 b_3  - b_3^2 - 2 b_2^2 + 1 \vert \leq   12 \alpha^4 - 52 \alpha^3 + 91 \alpha^2 - 74 \alpha + 24, $$
      where
\begin{align*}
   b_3 = \frac{ l_z (D^3 G(0) (z^3))}{3! \vert\vert z \vert\vert^3} \;\; \text{and} \;\;  b_2 &= \frac{ l_z (D^2 G(0) (z^2))}{2! \vert\vert z \vert\vert^2}, \quad l_z \in T_z.
\end{align*}
       The estimate is sharp.
\end{corollary}
\begin{corollary}
      Let $g \in \mathcal{H}(\mathbb{B}, \mathbb{C})$ with $g(0)=1$ and $G(z) = z g(z) \in \mathcal{SS}^*_\gamma(\mathbb{B})$. Then for $\gamma \in [1/3,1]$, the following holds:
      $$  \vert 2 b_2^2 b_3  - b_3^2 - 2 b_2^2 + 1 \vert \leq   15 \gamma^4 + 8 \gamma^2 +1, $$
      where
\begin{align*}
   b_3 = \frac{ l_z (D^3 G(0) (z^3))}{3! \vert\vert z \vert\vert^3} \;\; \text{and} \;\;  b_2 &= \frac{ l_z (D^2 G(0) (z^2))}{2! \vert\vert z \vert\vert^2}, \quad l_z \in T_z.
\end{align*}
       The estimate is sharp.
\end{corollary}
       When $\mathbb{B} = \mathbb{U}^n$ and $X = \mathbb{C}^n$, we obtain the following bounds:
\begin{corollary}
   Let $g \in \mathcal{H}( \mathbb{U}^n, \mathbb{C})$ with $g(0) =1$ and $G(z) = z g(z) \in \mathcal{S}^*(\mathbb{B})$. Then the following holds:
   $$  \| 2 b_2^2 b_3  - b_3^2 - 2 b_2^2 + 1  \|
      \leq   15 {\| z\|^6}  + 8  \| z \|^4 +1  , \; z \in \mathbb{U}^n,  $$
      where
\begin{align*}
       b_3 = \frac{ D^3 G(0) (z^3)}{3! }   \;\; \text{and} \;\; b_2^2 =  \frac{1}{2} D^2  G (0)  \bigg( z, \frac{D^2 G(0) (z^2)}{2!}\bigg).
\end{align*}
       The estimate is sharp.
\end{corollary}

\begin{corollary}
     Let $g \in \mathcal{H}( \mathbb{U}^n, \mathbb{C})$ with $g(0) =1$ and $G(z) = z g(z) \in \mathcal{S}^*_\alpha(\mathbb{B})$. Then for $\alpha \in [0,2/3]$, the following holds:
   $$   \| 2 b_2^2 b_3  - b_3^2 - 2 b_2^2 + 1 \| \leq   (1 - \alpha)^2 ( 2 \alpha -3 ) ( 6 \alpha -5) \| z \|^6  + 8 (1 - \alpha)^2  \| z\|^4 + 1, \quad z\in \mathbb{U}^n, $$
   where
\begin{align*}
       b_3 = \frac{ D^3 G(0) (z^3)}{3! }   \;\; \text{and} \;\; b_2^2 =  \frac{1}{2} D^2  G (0)  \bigg( z, \frac{D^2 G(0) (z^2)}{2!}\bigg).
\end{align*}
   The estimate is sharp.
\end{corollary}
\begin{corollary}
      Let $g \in \mathcal{H}(\mathbb{U}^n, \mathbb{C})$ with $g(0)=1$ and $G(z) = z g(z) \in \mathcal{SS}^*_\gamma(\mathbb{B})$. Then for $\gamma \in [1/3,1]$, the following holds:
      $$  \vert 2 b_2^2 b_3  - b_3^2 - 2 b_2^2 + 1 \vert \leq   15 \gamma^4 \|z\|^6 + 8 \gamma^2  \|z\|^4 +1, $$
      where
\begin{align*}
     b_3 = \frac{ D^3 G(0) (z^3)}{3! }   \;\; \text{and} \;\; b_2^2 =  \frac{1}{2} D^2  G (0)  \bigg( z, \frac{D^2 G(0) (z^2)}{2!}\bigg).
\end{align*}
       The estimate is sharp.
\end{corollary}
\section*{Declarations}
\subsection*{Funding}
The work of Surya Giri is supported by University Grant Commission, New Delhi, India  under UGC-Ref. No. 1112/(CSIR-UGC NET JUNE 2019).
\subsection*{Conflict of interest}
	The authors declare that they have no conflict of interest.
\subsection*{Author Contribution}
    Each author contributed equally to the research and preparation of the manuscript.
\subsection*{Data Availability} Not Applicable.
\noindent

\end{document}